\documentclass[smallextended]{svjour3}
\usepackage{amsmath,amssymb,bm}
\usepackage{graphicx}
\smartqed  

\newcommand{\T}{\mathcal{T}}
\newcommand{\E}{\mathcal{E}}
\newcommand{\ang}[1]{\langle #1 \rangle_{\partial\T_h}}
\newcommand{\rd}{\partial}
\newcommand{\wh}{\widehat}

\newcommand{\R}{\mathbb{R}}
\newcommand{\Id}{I}
\newcommand{\qflux}{\widehat{\bm q}_h}
\newcommand{\PV}{\bm{P}_{\bm V}}
\newcommand{\Pik}{\bm \Pi_k}
\newcommand{\PW}{P_W}
\newcommand{\PM}{P_M}
\newcommand{\Eq}{\bm e_{\bm q}}
\newcommand{\Eu}{e_u}
\newcommand{\Et}{e_{\wh u}}
\newcommand{\dVq}{\bm \delta_{\bm V}\bm q}
\newcommand{\dWu}{\delta_W u}
\newcommand{\dMu}{\delta_M u}
\newcommand{\dVth}{\bm \delta_{\bm V}\bm\theta}
\newcommand{\dWpsi}{\delta_W\psi}
\newcommand{\dMpsi}{\delta_M\psi}

\begin{document}
\title{An HDG method with orthogonal projections in facet integrals
\thanks{This work was supported by JSPS KAKENHI Numbers 15H03635, 15K13454, 17K14243 and 17K18738.}
}
\author{Issei Oikawa}
\institute{Issei Oikawa
\at
Waseda Research Institute for Science and Engineering, Waseda University,
 3-4-1 Okubo, Shinjuku-ku, Tokyo 169-8555, Japan\\
\email{oikawa@aoni.waseda.jp}}

\date{Received: date / Accepted: date}
\maketitle
\begin{abstract}
We propose and analyze a new hybridizable discontinuous Galerkin (HDG) method for second-order elliptic problems.
Our method is obtained by inserting the $L^2$-orthogonal
projection onto the approximate space for a numerical trace
into all facet integrals
 in the usual HDG formulation.
The orders of convergence for all variables are  optimal
  if we use polynomials of degree $k+l$, $k+1$ and $k$, where
  $k$ and $l$ are any non-negative integers,  to approximate the vector, scalar and trace variables, which implies that our method can achieve superconvergence
for the scalar variable without postprocessing.
Numerical results are  presented to verify the theoretical results.
\end{abstract}

\section{Introduction}
In this paper, we propose a new hybridizable discontinuous Galerkin (HDG) method for second-order elliptic problems. For simplicity, the following diffusion problem is considered:
\begin{subequations} \label{eq:pois}
\begin{align}
  \bm q + \nabla u &= 0  \quad \text{ in } \Omega, \\
  \nabla \cdot \bm q  &= f  \quad \text{ in } \Omega, \\
  u &= 0 \quad \text{ on } \rd\Omega,
\end{align}
\end{subequations}
where $\Omega \subset \R^d~(d=2,3)$ is a bounded and convex polygonal or polyhedral domain and $f$ is a given $L^2$-function.

To begin with, let us define notations for the description of the standard HDG method.
Let $\T_h$ be a mesh of $\Omega$, which consists of polygons or polyhedrons, where $h$ stands for the mesh size. Let
  $\E_h$ denote the set of faces of all elements in $\T_h$.
A family of meshes $\{\T_h\}_h$ is assumed to satisfy the chunkiness condition \cite{BrSc2008},
under which the trace and inverse inequalities hold.
We use the usual notation of the Sobolev spaces~\cite{AdFo2003}, such as $H^m(D)$,
$\|w\|_{m,D} := \|w\|_{H^m(D)}$, $|w|_{m,D} := |w|_{H^m(D)}$ for an integer $m$
and a domain $D \subset \R^d$.
When $D=\Omega$ or $m=0$, we omit the domain or the index and
simply write
$\|w\|_m = \|w\|_{m,\Omega}$, $|w|_m = |w|_{m,\Omega}$
 and $\|w\| = \|w\|_{0,\Omega}$.
 The piecewise or broken Sobolev space of order $m$ is defined by
  $H^m(\T_h) := \{ v \in L^2(\Omega) : v|_K \in H^m(K)~ \forall K \in \T_h\}$.
We denote by $L^2(\E_h)$ the $L^2$ space on the union
of all faces of $\E_h$
and by $P_k(\T_h)$ the space of piecewise polynomials  of degree $k$.
The piecewise inner products are written as
\[
  (u,v)_{\T_h} = \sum_{K \in \T_h} \int_K uv dx,
  \qquad
  \ang{u,v} = \sum_{K \in \T_h} \int_{\rd K} uv ds.
\]
The induced piecewise norm are denoted as $\|v\|_{\T_h} = (v,v)^{1/2}_{\T_h}$ and $\|v\|_{\rd\T_h} = \ang{v,v}^{1/2}$, and the piecewise Sobolev
 seminorm is defined by $|w|_{1,\T_h} =
 \left(\sum_{K \in \T_h} |w|_{1, K}^2\right)^{1/2}$.

Throughout the paper, we will use the symbol $C$ to
denote generic constants independent of $h$.
Vector variables and function spaces are displayed in boldface, such as $\bm P_k(\T_h) = P_k(\T_h)^d$.

We define finite element spaces for $\bm q$, $u$ and the trace of $u$ by
\begin{align*}
 &  \bm V_h =\{ \bm v\in \bm L^2(\Omega) : \bm v|_K \in \bm V(K)~\forall K \in \T_h\}, \\
& W_h = \{ w \in L^2(\Omega) : w |_K \in P_{k+1} (K)~\forall K \in \T_h\}, \\
& M_h = \{ \mu \in L^2(\E_h) : \mu|_F \in P_k(F) ~\forall F \in \E_h\},
\end{align*}
respectively, where $\bm V(K)$ is a finite-dimensional spaces satisfying $P_k(K) \subset \bm V(K)$.
 The $L^2$-orthogonal projections onto $\bm V_h$, $W_h$ and $M_h$
 are denoted by $\PV$, $\PW$ and $\PM$, respectively.
 We simply write  $\PM w = \PM(w|_{\E_h})$ for $w \in H^2(\T_h)$.
 Note that $\PM w$ may not belong to $M_h$
  since it is double-valued in general.

The standard HDG method reads as follows:
Find $(\bm q_h, u_h, \wh u_h)  \in \bm V_h \times W_h \times M_h$
 such that
\begin{subequations} \label{std-hdg}
\begin{align}
	(\bm q_h, \bm v)_{\T_h} -(u_h, \nabla \cdot \bm v)_{\T_h}
	 + \ang{\wh u_h, \bm v \cdot \bm n} &= 0
	 & \forall \bm v \in \bm V_h, \label{std-hdg-a}  \\
	-(\bm q_h, \nabla w)_{\T_h} + \ang{\qflux\cdot \bm n, w}  &= (f, w)
	 & \forall w \in W_h, \label{std-hdg-b}\\
	\ang{\qflux \cdot \bm n, \mu}& = 0
	&\forall \mu \in M_h, \label{std-hdg-c}
\end{align}
\end{subequations}
where $\qflux$ is the numerical flux defined by
\begin{align} \label{qflux}
  \qflux \cdot \bm n= \bm q_h \cdot \bm n
	 + \tau (u_h - \wh u_h).
\end{align}
Here, $\tau$ is a positive parameter and
 is set to be of order $O(h^{-1})$ in the paper.

The so-called Lehrenfeld-Sch\"oberl (LS) numerical flux \cite{Lehrenfeld2010} is obtained by inserting $\PM$ into the stabilization part of the numerical flux:
\begin{align*} 
  \qflux^{LS} \cdot \bm n= \bm q_h \cdot \bm n
	 + \tau (\PM u_h - \wh u_h).
\end{align*}
In \cite{Oikawa2015}, it was proved that the HDG method using the LS numerical flux (the HDG-LS method)
achieves optimal-order convergence for all variables
 if we use polynomials of degree $k$, $k+1$ and $k$ for $\bm V_h$, $W_h$ and $M_h$, respectively.
 It can be said that the HDG-LS method is superconvergent for the scalar variable $u$ without postprocessing.
  Another more elaborate flux was introduced in the hybrid higher-order (HHO) method \cite{PEL2014,PiAl2015}, which
 was recently linked to the HDG method in \cite{CPE2016}.
 The LS numerical flux approach has been applied to various problems;
 linear elasticity~\cite{QJS2018},
 convection-diffusion problems~\cite{QiSh2016}, Stokes equations~\cite{Oikawa2016}, Navier-Stokes equations~\cite{QiSh2016ns,LeSc2016} and Maxwell's equations~\cite{CQSS2017}.

 Let us here point out that the superconvergence of the HDG-LS method is sensitive to the choice of $\bm V_h$.
For example, the superconvergence property is no longer maintained
 if  $\bm V_h$ is taken to be $\bm P_{k+1}(\T_h)$ instead of $\bm P_{k}(\T_h)$.
 We now demonstrate that by numerical experiments for
 the test problem \eqref{testprob} which will be provided in Section 4.
The numerical results
are shown in Table \ref{tbl1}.
In the case of $\bm V_h = \bm P_1(\T_h)$,
the orders of convergence
are optimal for both variables $\bm q$ and $u$.
On the other hand,
all the orders become sub-optimal for $\bm V_h = \bm P_2(\T_h)$.
 \begin{table}[h]
 \caption{Convergence history of the HDG-LS method for $\bm V_h \times W_h \times M_h = \bm P_l(\T_h) \times P_2(\T_h) \times P_1(\E_h)$}
 \label{tbl1}
 \centering
 \begin{tabular}{|c|c|cc|cc|cc|} \hline
  & & \multicolumn{2}{c|}{$\|\bm q - \bm q_h\|$} &
  \multicolumn{2}{c|}{$\|u-u_h\|$} &
  \multicolumn{2}{c|}{$\|h^{-1/2}(\PM u_h - \wh u_h)\|_{\rd \T_h}$}   \\
$l$ & $1/h$  & Error & Order & Error & Order & Error & Order  \\
\hline
 & 10 & 1.236E-02 & -- & 7.400E-04 & -- & 2.719E-02 & -- \\
1 & 20 & 3.083E-03 & 2.00 & 9.085E-05 & 3.03 & 6.676E-03 & 2.03 \\
 & 40 & 7.655E-04 & 2.01 & 1.140E-05 & 2.99 & 1.662E-03 & 2.01 \\
 & 80 & 1.915E-04 & 2.00 & 1.414E-06 & 3.01 & 4.113E-04 & 2.01 \\
 \hline
 & 10 & 1.464E-01 & -- & 2.738E-03 & -- & 1.064E-02 & -- \\
2 & 20 & 7.177E-02 & 1.03 & 6.517E-04 & 2.07 & 4.999E-03 & 1.09 \\
 & 40 & 3.543E-02 & 1.02 & 1.568E-04 & 2.06 & 2.363E-03 & 1.08 \\
 & 80 & 1.744E-02 & 1.02 & 3.815E-05 & 2.04 & 1.180E-03 & 1.00 \\ \hline
\end{tabular}
\end{table}
\par
 The aim of the paper is to recover
  the superconvergence property for such cases.
  The key idea is to insert the orthogonal projection $\PM$ into the facet integrals in the usual HDG formulation.
  The resulting method
  can achieve optimal convergence in $\bm q$ and
  superconvergence in $u$ without postprocessing
  if $\bm V_h$ contains $\bm P_k(\T_h)$, see Theorems \ref{thm1} and
  \ref{thm2}.

 The rest of the paper is organized as follows.
 In Section 2, we introduce a new HDG method.
 In Section 3,  error estimates for both variables $u$ and $\bm q$ are provided.
 Numerical results are presented to verify our theoretical results in Section 4.
\section{An HDG method with orthogonal projections}
We begin by introducing our method:
 Find $(\bm q_h, u_h, \wh u_h) \in \bm V_h \times W_h\times M_h$
 such that
\begin{subequations} \label{newhdg}
\begin{align}
(\bm q_h + \nabla u_h, \bm v)_{\T_h}
 - \ang{\PM u_h - \wh u_h, \bm v \cdot \bm n} &= 0,
 & \forall\bm v\in \bm V_h, \label{newhdg-a}\\
 -(\bm q_h, \nabla w)_{\T_h} + \ang{\qflux \cdot \bm n, \PM w}  &= (f, w),
 & \forall w \in W_h,
 \label{newhdg-b}
 \\
 \ang{\qflux \cdot \bm n, \mu}& = 0,
 & \forall \mu \in M_h, \label{newhdg-c}
\end{align}
\end{subequations}
where $\qflux$ is the standard numerical flux defined by \eqref{qflux}.
The derivation of our method is simple.
Integrating by parts in \eqref{std-hdg-a} and
replacing $u_h$ by $\PM u_h$, we get \eqref{newhdg-a}.
The second equation \eqref{newhdg-a} is obtained by
 replacing $w$ by $\PM w$ in \eqref{std-hdg-b}.
The third equation \eqref{newhdg-c} is just the same as \eqref{std-hdg-c}.
Since $\mu = \PM \mu $ in $\eqref{newhdg-c}$ (and $\eqref{std-hdg-c}$),
 we can also consider that our method is obtained by inserting the orthogonal projection $\PM$ in all facet integrals in
 the standard HDG method.
\begin{remark}
If $\bm v \cdot \bm n|_F \in P_k(F)$ for any  $F \in \E_h$, then our method is identical to the HDG-LS method since
\begin{align*}
 &\ang{\PM u_h, \bm v \cdot \bm n} = \ang{u_h, \bm v   \cdot \bm n}
&  \text{ in } \eqref{newhdg-a}, \\
 &\ang{\qflux   \cdot \bm n, \PM w} =
  \ang{\PM(\qflux\cdot\bm n), w} = \ang{\qflux^{LS} \cdot \bm n, w}
 & \text{ in } \eqref{newhdg-b}, \\
&\ang{\qflux \cdot \bm n, \mu} = \ang{\PM(\qflux \cdot \bm n), \mu}
=\ang{\qflux^{LS} \cdot \bm n, \mu}
  &\text{ in } \eqref{newhdg-c}.
\end{align*}
\end{remark}
\section{Error analysis}
In this section, we provide
 the optimal-order error estimates of our method.
We are going to use the following approximation properties:
\begin{align*}
\| \bm v - \PV \bm v\|
  &\le C h^{j} |\bm v|_{j},
  & 1 \le j \le k+1,\\
\| \bm v - \PV \bm v\|_{\rd\T_h}
  &\le C h^{j-1/2} |\bm v|_{j},
  &1 \le j \le k+1,\\
\| \nabla (w - \PW w) \|_{\T_h}
  &\le C h^{j-1} |w|_{j},
  & 1 \le j \le k+2,\\
\| w - \PW w \| & \le C h^{j} |w|_{j},
  &  1 \le j \le k+2,\\
\| w - \PW w \|_{\rd\T_h}  &\le C h^{j-1/2} |w|_{j},
  & 1 \le j \le k+2,\\
\| w - \PM w \|_{\rd\T_h}
 &\le C h^{j-1/2} |w|_{j},
 & 1 \le j \le k+1,
\end{align*}
for any $\bm v \in \bm H^j(\Omega)$ and $w \in H^j(\Omega)$.
For the piecewise Sobolev spaces, the following hold:
\begin{align}
 \| w - \PM w\|_{\rd\T_h}
 &\le Ch^{1/2}|w|_{1,\T_h}
 &&\forall w \in H^1(\T_h),
 \label{ineq-PMw}\\
 \| \bm v \cdot \bm n- \PM (\bm v \cdot \bm n)\|_{\rd\T_h}
 &\le Ch^{1/2}|\bm v|_{1,\T_h}
 && \forall \bm v \in \bm H^1(\T_h).
  \label{ineq-PMvn}
\end{align}
Let $\Pik$ be the orthogonal projection from  $\bm H^1(\T_h)$
onto $\bm P_k(\T_h)$, which satisfies
\begin{align}
 &\Pik \bm v \cdot \bm n |_{\rd K} \subset P_k(\rd K) \quad \forall K \in \T_h,
  \label{A1}  \\
 &\| \bm v \cdot \bm n- \Pik \bm v\cdot \bm n\|_{\rd\T_h} \le Ch^{j-1/2}|\bm v|_{j}
 \quad \text{ for } \bm v \in \bm H^{j}(\Omega), 1\le j \le k+1.
\label{A2}
\end{align}
The insertion of $\PM$ in \eqref{newhdg-b}
 gives rise to some terms in the form
\[
 R(\bm v, w) := \ang{(\Id - \PM)\bm v\cdot\bm n, w}
\]
in error analysis.
We show the bound of $R(\cdot,\cdot)$
 by the properties \eqref{A1} and \eqref{A2}.
\begin{lemma} \label{lem:pik}
For all $\bm v \in \bm H^{k+1}(\Omega)$ and $w \in H^1(\T_h)$, we have
\[
  |R(\bm v, w)| \le Ch^{k+1}|\bm v|_{k+1} |w|_{1,\T_h}.
\]
\end{lemma}
\begin{proof}
By \eqref{A1}, \eqref{A2} and \eqref{ineq-PMw}, we have
\begin{align*}
  |R(\bm v, w)|
  &=|\ang{(\Id - \PM)(\bm v - \Pik \bm v) \cdot \bm n, (\Id - \PM)w}|\\
  &\le \|\bm v\cdot \bm n - \Pik \bm v\cdot \bm n\|_{\rd\T_h}\|(\Id - \PM)w\|_{\rd\T_h} \\
 &\le Ch^{k+1}|\bm q|_{k+1} |w|_{1,\T_h}.
\end{align*}
This completes the proof. \qed
\end{proof}
\subsection{Error equations}
As a lemma, we show the error equations in terms of the projections of the errors:
\[
 \Eq =  \PV\bm q - \bm q_h, \quad
 \Eu = \PW u - u_h, \quad
 \Et = \PM u - \wh u_h.
\]
The approximation errors are denoted as
\[
  \dVq = \bm q - \PV\bm q, \quad
  \dWu = u - \PW u, \quad
 \dMu = u - \PM u.
\]
\begin{lemma} \label{lem:erreq}
 The following equations hold:
\begin{subequations} \label{errequs}
\begin{align}
	(\Eq, \bm v)_{\T_h}
	+(\nabla \Eu, \bm v)_{\T_h}
	-\ang{\PM \Eu - \Et, \bm v \cdot \bm n}
	&= F_1(\bm v)
	& \forall\bm v\in \bm V_h, \label{err-a}\\
	-(\Eq, \nabla w)
  + \ang{\wh\Eq \cdot \bm n, \PM w}  &= F_2(w)
	& \forall w \in W_h,
	\label{err-b}
	\\
	\ang{\wh\Eq \cdot \bm n, \mu}& = F_3(\mu)
	& \forall \mu \in M_h, \label{err-c}
\end{align}
\end{subequations}
where $\wh\Eq \cdot \bm n
= \Eq \cdot\bm n + \tau(\PM \Eu - \Et)$ and
\begin{align*}
& F_1(\bm v)
  = -(\nabla \dWu, \bm v)_{\T_h}
    + \ang{\PM \dWu, \bm v \cdot \bm n},\\
&F_2(w)
  = -R(\bm q, w)
    -\ang{\dVq \cdot \bm n - \tau  \dWu, \PM w},\\
&F_3(\mu)
 = -\ang{\dVq \cdot \bm n - \tau \dWu, \mu}.
\end{align*}
\end{lemma}
\begin{proof}
We easily see that the exact solution satisfies
\begin{subequations} \label{cons}
\begin{align}
  (\bm q, \bm v)_{\T_h}
  +(\nabla u, \bm v)_{\T_h}
  &= 0
  &\forall \bm v \in \bm V_h,  \label{cons-a}\\
 -(\bm q, \nabla w)_{\T_h}
 +\ang{\bm q \cdot \bm n,  w} &= (f,w)
 & \forall w \in W_h, \label{cons-b} \\
 \ang{\bm q \cdot \bm n, \mu} &= 0
 & \forall \mu \in M_h.  \label{cons-c}
\end{align}
\end{subequations}
Each term in  \eqref{cons} is rewritten
in terms of $\PV q, \PW u$ and $\PM u$ as follows:
\begin{align*}
 (\bm q, \bm v)_{\T_h} &= (\PV\bm q, \bm v)_{\T_h},\\
 (\nabla u, \bm v)_{\T_h}
 &= (\nabla \PW u, \bm v)_{\T_h}
+ (\nabla \dWu, \bm v)_{\T_h}, \\
 (\bm q, \nabla w)_{\T_h} &= (\PV\bm q, \nabla w)_{\T_h}, \\
\ang{\bm q \cdot \bm n,  w}
&=\ang{\bm q \cdot \bm n,  \PM w} + R(\bm q, w) \\
&= \ang{\PV \bm q \cdot \bm n,  \PM w}
+\ang{\dVq \cdot \bm n,  \PM w} + R(\bm q, w), \\
\ang{\bm q \cdot \bm n, \mu}
 &= \ang{\PV \bm q \cdot \bm n,  \mu}
 + \ang{\dVq \cdot \bm n, \mu}.
\end{align*}
%
Taking the stabilization terms into account, we have
\begin{subequations}
\label{proju-eq}
\begin{align}
  (\PV\bm q+\nabla \PW u, \bm v)_{\T_h}
  - \ang{\PM(\PW u) - \PM u, \bm v \cdot \bm n}
  &= F_1(\bm v)
  &\forall \bm v \in \bm V_h, \\
 -(\PV\bm q, \nabla w)_{\T_h}
 +\ang{\wh{\PV\bm q} \cdot \bm n , \PM w}
  &=
  (f,w) + F_2(w)
 & \forall w \in W_h, \\
 \ang{\wh{\PV\bm q} \cdot \bm n , \mu} &= F_3(\mu)
 & \forall \mu \in M_h,
\end{align}
\end{subequations}
where $\wh{\PV\bm q}\cdot \bm n := \PV \bm q \cdot \bm n + \tau(\PM u - \PW u)$.
Subtracting \eqref{newhdg} from \eqref{proju-eq}, we obtain the required equations. \qed
\end{proof}
From Lemma \ref{lem:erreq}, the below inequalities follow.
\begin{lemma} \label{lem:Eq-GradEu}
If $u \in H^{k+2}(\Omega)$, then we have
\begin{equation} \label{Eq-GradEu-a}
\|\nabla \Eu\|_{\T_h} \le
C\left(
 \|\Eq\| + h^{-1/2}\|\PM\Eu - \Et\|_{\rd\T_h} + h^{k+1}|u|_{k+2}
\right)
\end{equation}
and
\begin{equation} \label{Eq-GradEu-b}
\|\Eq\| \le C\left(
 \|\nabla\Eu\|_{\T_h} + h^{-1/2}\|\PM\Eu - \Et\|_{\rd\T_h}
 +h^{k+1}|u|_{k+2}
\right).
\end{equation}
\end{lemma}
\begin{proof}
Taking $\bm v = \nabla \Eu$ in \eqref{err-a},
we have
\begin{align*}
\|\nabla \Eu\|^2_{\T_h}
 &= -(\Eq, \nabla \Eu)_{\T_h}
 	 +\ang{\PM\Eu-\Et, \nabla\Eu\cdot\bm n}
 + F_1( \nabla \Eu).
\end{align*}
The first two terms on the right-hand side
are bounded as
\begin{align*}
&|(\Eq, \nabla \Eu)_{\T_h}|
 \le \|\Eq\|\|\nabla \Eu\|_{\T_h},
 \\
&|\ang{\PM \Eu-\Et, \nabla \Eu \cdot \bm n}|
\le Ch^{-1/2}\|\PM \Eu - \Et \|_{\rd\T_h} \|\nabla \Eu\|_{\T_h}.
\end{align*}
By the inverse inequality, we can estimate the remaining term as
\begin{equation} \label{est-f1}
\begin{aligned}
|F_1(\nabla \Eu)|
 &\le |(\nabla \dWu, \nabla \Eu)_{\T_h}|
 +|\ang{\PM\dWu, \nabla\Eu\cdot\bm n}|
  \\
 &\le Ch^{k+1}|u|_{k+2}\|\nabla \Eu\|_{\T_h}.
\end{aligned}
\end{equation}
Combining these results yields \eqref{Eq-GradEu-a}.
Similarly, we can prove \eqref{Eq-GradEu-b}. \qed
\end{proof}
\subsection{The estimate of $\Eq$}
We are now ready to show the main results of the paper.
\begin{theorem} \label{thm1}
If $u \in H^{k+2}(\Omega)$, then we have
\begin{align*}
 \|\Eq\|
 + \|\tau^{1/2}(\PM \Eu- \Et)\|_{\rd\T_h} \le C h^{k+1}|u|_{k+2}.
\end{align*}
\end{theorem}
\begin{proof}
Substituting $w = e_u$ in \eqref{err-b}
and $\mu = e_{\wh u}$ in \eqref{err-c}, we have
\begin{align}
 -(\Eq, \nabla \Eu)_{\T_h}
 + \ang{\wh\Eq \cdot \bm n, \PM \Eu - \Et}
     = F_2(e_u) - F_3(e_{\wh u}).
   \label{err-b-c}
\end{align}
Taking $\bm v = \bm e_{\bm q}$ in \eqref{err-a} and adding it
to \eqref{err-b-c}, we get
\begin{align*}
 \|\bm e_{\bm q}\|^2
 +\|\tau^{1/2}(\PM \Eu - \Et)\|_{\rd\T_h}^2
 = F_1(\bm e_{\bm q}) + F_2(e_u) - F_3(e_{\wh u}).
\end{align*}
In a similar way to \eqref{est-f1}, we have
\begin{align*}
|F_1(\Eq)|
 &\le Ch^{k+1}|u|_{k+2} \|\Eq\|.
\end{align*}
The rest terms are written as
\begin{equation*}
\begin{aligned}
F_2(\Eu) - F_3(\Et)
 &=
 -R(\bm q, \Eu)
 +\ang{\dVq\cdot\bm n-\tau\PM\dWu, \PM\Eu-\Et} \\
 & =: I_1 + I_2.
\end{aligned}
\end{equation*}
By Lemmas \ref{lem:pik} and \ref{lem:Eq-GradEu}, we have
\begin{align*}
|I_1|
  &\le Ch^{k+1}|u|_{k+2} \|\nabla \Eu\|_{\T_h} \\
  &\le  Ch^{k+1} |u|_{k+2}
  \left(
  \|\Eq\|+h^{-1/2}\|\PM\Eu-\Et\|_{\rd\T_h}+h^{k+1}|u|_{k+2}
 \right).
\end{align*}
The term $I_2$ is bounded as, in view of $\tau=O(h^{-1})$,
\begin{align*}
 |I_2|
  &\le (\|\dVq\|_{\rd\T_h} + \tau\|\dWu\|_{\rd\T_h})
   \|\PM \Eu - \Et\|_{\rd\T_h}\\
  &= Ch^{k+1}|u|_{k+2}\cdot\tau^{1/2} \|\PM \Eu - \Et\|_{\rd\T_h}.
\end{align*}
Using Young's inequality and
 arranging the terms, we obtain
\begin{equation*} 
\begin{aligned}
&\|\Eq\|^2
 + \|\tau^{1/2}(\PM \Eu- \Et)\|_{\rd\T_h}^2 \le
 C h^{2(k+1)} |u|_{k+2}^2,
\end{aligned}
\end{equation*}
which completes the proof. \qed
\end{proof}
\subsection{The estimate of $\Eu$}
We show that the order of convergence in the variable $u$ is optimal by the duality argument.
To this end, we consider the adjoint problem of \eqref{eq:pois}:
Find $\psi \in H^2(\Omega)\cap H^1_0(\Omega)$
and $\bm\theta \in \bm H^1(\Omega)$ such that
\begin{align*}
   \nabla \psi + \bm \theta &= 0  & \text{ in } \Omega, \\ 
   \nabla \cdot \bm \theta &= \Eu & \text{ in } \Omega, \\
   \psi &=  0                     & \text{ on } \rd\Omega. 
\end{align*}
As is well known, the  elliptic regularity holds:
\[
   \|\bm\theta\|_1 + \|\psi\|_2 \le C\|\Eu\|.
\]
Let us denote the approximation errors of $\psi$ and $\bm \theta$ as follows:
\[
  \dVth = \bm\theta - \PV\bm\theta, \quad
  \dWpsi = \psi - \PW \psi, \quad
  \dMpsi = \psi - \PM \psi.
\]
\begin{theorem} \label{thm2}
If $u \in H^{k+2}(\Omega)$, then we have
\[
  \| \Eu\| \le Ch^{k+2} |u|_{k+2}.
\]
\end{theorem}
\begin{proof}
Similarly to \eqref{proju-eq}, we deduce
\begin{subequations}
\begin{align}
  (\PV\bm \theta+ \nabla \PW \psi, \bm v)_{\T_h}
  -
 \ang{\PM (\psi -  \PW \psi), \bm v \cdot \bm n}
  &= G_1(\bm v)
  &\forall \bm v \in \bm V_h, \label{acons-a}\\
 -(\PV\bm \theta, \nabla w)_{\T_h}
 +\ang{\wh{\PV\bm \theta} \cdot \bm n , \PM w}
  &=
  (\Eu,w)+G_2(w)
 & \forall w \in W_h, \label{acons-b}\\
 \ang{\wh{\PV\bm \theta} \cdot \bm n , \mu} &= G_3(\mu)
 & \forall \mu \in M_h,  \label{acons-c}
\end{align}
\end{subequations}
where $\wh{\PV\bm\theta}\cdot\bm n
 = \PV\bm \theta\cdot \bm n + \tau(\PW \psi - \PM \psi)$ and
\begin{align*}
& G_1(\bm v)
  =
  -(\nabla \dWpsi, \bm v)_{\T_h}
  +\ang{\PM \dWpsi, \bm v\cdot\bm n},\\
&G_2(w)
  =
  -R(\bm\theta, w)
   -\ang{ \dVth \cdot \bm n - \tau\PM \dWpsi, \PM w},\\
&G_3(\mu)
 =-\ang{\dVth \cdot \bm n - \tau \PM\dWpsi, \mu}.
\end{align*}
Substituting $\bm v = \Eq$ in \eqref{acons-a},
 $w = \Eu$ in \eqref{acons-b} and $\mu = \Et$ in \eqref{acons-c} yields
\begin{subequations}\label{cons-th}
\begin{align}
  &(\bm \theta+\nabla \PW \psi, \Eq)_{\T_h}
  -\ang{\PM (\psi-\PW\psi), \Eq \cdot \bm n}
  = G_1(\Eq),
  \label{cons-th-a}
\\
 &-(\bm \theta, \nabla \Eu)_{\T_h}
  +\ang{\wh{\PV\bm \theta} \cdot \bm n , \PM\Eu - \Et}
  = \|\Eu\|^2+G_2(\Eu) - G_3(\Et).
  \label{cons-th-b}
\end{align}
\end{subequations}
Taking $\bm v = \PV\bm\theta$,
 $w  = \PW \psi$ and $\mu = \PM \psi$ in
the error equations \eqref{errequs}, we have
\begin{subequations} \label{errequs-theta}
\begin{align}
	(\Eq+\nabla \Eu, \bm\theta)_{\T_h}
	 -\ang{\PM \Eu - \Et, \PV\bm\theta \cdot \bm n}
	 &= F_1(\PV\bm\theta),
   \label{errequs-theta-a} \\
	-(\Eq, \nabla \PW \psi)
	+ \ang{\PM(\wh\Eq \cdot \bm n), \PW \psi - \psi}
	  &= F_2(\PW \psi) - F_3(\PM\psi).
    \label{errequs-theta-b}
\end{align}
\end{subequations}
Note that $\PM\psi \in M_h$ since $\psi$ is single-valued on $\E_h$.
Adding \eqref{errequs-theta-b} to \eqref{cons-th-a}  and
\eqref{errequs-theta-a} to \eqref{cons-th-b}, we have
\begin{align*}
(\bm \theta, \Eq)_{\T_h}
- \ang{\tau(\PM \Eu - \Et), \dWpsi}
 &= G_1(\Eq) + F_2(\PW\psi) - F_3(\PM\psi), \\
 (\Eq,\bm \theta)_{\T_h}
 - \ang{\tau(\PM \Eu - \Et), \dWpsi}
 &= \|\Eu\|^2+F_1(\PV\bm\theta) + G_2(\Eu) - G_3(\Et),
\end{align*}
respectively. Since  the left-hand sides are equal to each other,
we obtain
\begin{align*}
 \|\Eu\|^2
 = G_1(\Eq) - G_2(\Eu) + G_3(\Et) -
   (F_1(\PV\bm\theta) - F_2(\PW \psi) + F_3(\PM \psi)).
\end{align*}
By the inverse and trace inequalities, we have
\begin{align*}
 |G_1(\Eq)|
  &\le Ch|\psi|_{2}\|\Eq\|.
\end{align*}
By Lemma \ref{lem:Eq-GradEu} and Theorem \ref{thm1},
we get
\begin{align*}
|G_2(\Eu) - G_3(\Et)|
&\le |R(\bm\theta, \Eu)|
+  |\ang{-\dVth \cdot \bm n
   + \tau \dWpsi, \PM \Eu - \Et}|\\
&\le Ch|\bm\theta|_1 \|\nabla \Eu\|_{\T_h}
    + Ch(|\bm\theta|_1 + |\psi|_2) \cdot\tau^{1/2}\|\PM \Eu - \Et\|_{\rd\T_h} \\
&= Ch\|\Eu\|
( \|\Eq\|+  \tau^{1/2}\|\PM \Eu - \Et\| + h^{k+1}|u|_{k+2}) \\
&\le Ch^{k+2}|u|_{k+2}\|\Eu\|,
\end{align*}
Similarly, we have
\begin{align*}
F_1(\PV\bm\theta)
 &=
  -(\nabla\dWu, \bm\theta)
  +\ang{\dWu, \PM(\PV\bm\theta\cdot\bm n)}\\
&=  (\dWu, \nabla \cdot \bm\theta)
  -\ang{(\Id-\PM)\dWu, \PV\bm\theta\cdot\bm n}\\
&=: T_1+T_2,
\end{align*}
and the terms are bounded as
\begin{align*}
  |T_1| &\le \|\dWu\|\|\Eu\|\le Ch^{k+2}|u|_{k+2}\|\Eu\|,\\
  |T_2| &=
  |\ang{(\Id-\PM)\dWu, (\PV\bm\theta - \bm \theta
   + \bm\theta - \Pik\bm\theta)\cdot\bm n}|\\
  &\le \|\dWu\|_{\rd\T_h}
     (\|\dVth\|_{\rd\T_h}
      +\|\bm \theta - \Pik\bm\theta\|_{\rd\T_h}) \\
  &\le Ch^{k+3/2}|u|_{k+2} \cdot Ch^{1/2}|\bm\theta|_1 \\
  &\le Ch^{k+2}|u|_{k+2} |\bm\theta|_1.
\end{align*}
Moreover,
\begin{align*}
F_2(\PW \psi) - F_3(\PM\psi)
&=-R(\bm q, \PW\psi)
    -\ang{\dVq \cdot \bm n - \tau\PM\dWu, \PM \dWpsi}\\
&=: T_3 + T_4 .
\end{align*}
Since both $\bm q$ and $\psi$ are single-valued on $\E_h$, it follows that
\[
 T_3 = -\ang{(\Id-\PM) \bm q \cdot \bm n, \dWpsi}
 = R(\bm q, \PW \psi - \psi).
\]
By Lemma \ref{lem:pik}, we get
\begin{align*}
 |T_3|
 &\le C h^{k+1}|u|_{k+2} |\dWpsi|_{1, \T_h}  \le Ch^{k+2}|u|_{k+2} |\psi|_2.
\end{align*}
The other term is bounded as follows:
\begin{align*}
|T_4| &\le
 C\left(
  \|\dVq\|_{\rd\T_h} + \tau \|\dWu\|_{\rd\T_h}
 \right) \|\dWpsi\|_{\rd\T_h}\\
 &\le C(h^{k+1/2}|\bm q|_{k+1} + h^{-1}h^{k+3/2}|u|_{k+2}) \cdot Ch^{3/2}|\psi|_2 \\
 &\le Ch^{k+2}|u|_{k+2}|\psi|_2.
\end{align*}
Combining these results and applying Young's inequality, we have
\begin{align*}
&\|\Eu\|^2 \le
Ch^{k+2}\left(\|\Eu\|+|\bm\theta|_1+|\psi|_2 \right).
\end{align*}
Thanks to the elliptic regularity, we obtain the required inequality.
\qed
\end{proof}

\section{Numerical results}
In this section, we carry out numerical experiments
 to verify  our theoretical results.
 The following test problem is considered:
\begin{subequations} \label{testprob}
\begin{align}
   -\Delta u &= 2\pi^2 \sin(\pi x) \sin(\pi y) &&\text{ in } \Omega, \\
           u &= 0 && \text{ on } \rd\Omega,
\end{align}
\end{subequations}
where $\Omega =  (0,1)^2$ and the exact solution is $\sin(\pi x)\sin(\pi y)$.
 All computations were done with
  \texttt{FreeFem++}~\cite{FreeFem}.
The meshes we used are unstructured triangular meshes.
We set $\bm V_h = \bm P_{k+l}(\T_h)$, $W_h = P_{k+1}(\T_h)$
and $M_h = P_k(\E_h)$ for $0 \le k \le 2$,
varying $l$ from $0$ to $2$.
The stabilization parameter $\tau$ is set to be $1/h$ in all cases.

 The history of convergence of our method is displayed in Tables \ref{tbl:p0}--\ref{tbl:p2}.
From the results, we observe that
the orders or convergence in $\bm q$, $u$ and the projected jump
 quantity
 are  $k+1$, $k+2$ and $k+1$, respectively, which is in full agreement with
 Theorems \ref{thm1} and \ref{thm2}.
Note that, as mentioned in Remark 1, the errors of our method in Table \ref{tbl:p1} for $l=0$  coincide
with those of the HDG-LS method  in Table \ref{tbl1}.
\begin{table}[h]
  \caption{Convergence history for $k=0$}
  \label{tbl:p0}
\centering
  \begin{tabular}{|c|c|cc|cc|cc|} \hline
&  & \multicolumn{2}{c|}{$\|\bm q - \bm q_h\|$} &
   \multicolumn{2}{c|}{$\|u-u_h\|$} &
   \multicolumn{2}{c|}{$\|h^{-1/2}(\PM u_h - \wh u_h)\|_{\rd \T_h}$}   \\
$l$ & $1/h$  & Error & Order & Error & Order & Error & Order  \\
   \hline
  & 10 & 2.643E-01 & -- & 1.842E-02 & -- & 3.873E-01 & -- \\
0 & 20 & 1.306E-01 & 1.02 & 4.543E-03 & 2.02 & 1.920E-01 & 1.01 \\
  & 40 & 6.612E-02 & 0.98 & 1.176E-03 & 1.95 & 9.775E-02 & 0.97 \\
  & 80 & 3.313E-02 & 1.00 & 2.928E-04 & 2.01 & 4.853E-02 & 1.01 \\ \hline
  & 10 & 2.152E-01 & -- & 6.793E-03 & -- & 5.545E-02 & -- \\
1 & 20 & 1.069E-01 & 1.01 & 1.686E-03 & 2.01 & 2.673E-02 & 1.05 \\
  & 40 & 5.407E-02 & 0.98 & 4.339E-04 & 1.96 & 1.366E-02 & 0.97 \\
  & 80 & 2.726E-02 & 0.99 & 1.113E-04 & 1.96 & 6.748E-03 & 1.02 \\ \hline
  & 10 & 2.533E-01 & -- & 6.078E-03 & -- & 1.502E-02 & -- \\
2 & 20 & 1.254E-01 & 1.01 & 1.498E-03 & 2.02 & 7.143E-03 & 1.07 \\
  & 40 & 6.348E-02 & 0.98 & 3.847E-04 & 1.96 & 3.654E-03 & 0.97 \\
  & 80 & 3.184E-02 & 1.00 & 9.910E-05 & 1.96 & 1.802E-03 & 1.02 \\ \hline
\end{tabular}
\end{table}

\begin{table}[h]
  \caption{Convergence history for $k=1$}
     \label{tbl:p1}
\centering
\begin{tabular}{|c|c|cc|cc|cc|} \hline
 &  & \multicolumn{2}{c|}{$\|\bm q - \bm q_h\|$} &
 \multicolumn{2}{c|}{$\|u-u_h\|$} &
 \multicolumn{2}{c|}{$\|h^{-1/2}(\PM u_h - \wh u_h)\|_{\rd \T_h}$}   \\
$l$ & $1/h$  & Error & Order & Error & Order & Error & Order  \\
 \hline
  &  10 & 1.236E-02 & -- & 7.400E-04 & -- & 2.719E-02 & -- \\
0 &  20 & 3.083E-03 & 2.00 & 9.085E-05 & 3.03 & 6.676E-03 & 2.03 \\
  &  40 & 7.655E-04 & 2.01 & 1.140E-05 & 2.99 & 1.662E-03 & 2.01 \\
  &  80 & 1.915E-04 & 2.00 & 1.414E-06 & 3.01 & 4.113E-04 & 2.01 \\ \hline
  &  10 & 7.212E-03 & -- & 1.023E-04 & -- & 3.157E-03 & -- \\
1 &  20 & 1.744E-03 & 2.05 & 1.194E-05 & 3.10 & 7.704E-04 & 2.03 \\
  &  40 & 4.468E-04 & 1.96 & 1.544E-06 & 2.95 & 1.863E-04 & 2.05 \\
  &  80 & 1.127E-04 & 1.99 & 1.922E-07 & 3.01 & 4.525E-05 & 2.04 \\ \hline
 & 10 & 8.936E-03 & -- & 1.059E-04 & -- & 1.763E-03 & -- \\
2 & 20 & 2.187E-03 & 2.03 & 1.254E-05 & 3.08 & 4.422E-04 & 2.00 \\
 & 40 & 5.529E-04 & 1.98 & 1.616E-06 & 2.96 & 1.062E-04 & 2.06 \\
 & 80 & 1.387E-04 & 1.99 & 2.019E-07 & 3.00 & 2.594E-05 & 2.03 \\
\hline
\end{tabular}
\end{table}
\begin{table}[h]
  \caption{Convergence history for $k=2$}
\label{tbl:p2}
\centering
\begin{tabular}{|c|c|cc|cc|cc|} \hline
& & \multicolumn{2}{c|}{$\|\bm q - \bm q_h\|$} &
\multicolumn{2}{c|}{$\|u-u_h\|$} &
\multicolumn{2}{c|}{$\|h^{-1/2}(\PM u_h - \wh u_h)\|_{\rd \T_h}$}   \\
$l$ &$1/h$  & Error & Order & Error & Order & Error & Order  \\
\hline
 & 10 & 2.365E-04 & -- & 2.705E-05 & -- & 1.270E-03 & -- \\
0 & 20 & 2.718E-05 & 3.12 & 1.573E-06 & 4.10 & 1.480E-04 & 3.10 \\
 & 40 & 3.509E-06 & 2.95 & 1.049E-07 & 3.91 & 1.919E-05 & 2.95 \\
 & 80 & 4.350E-07 & 3.01 & 6.381E-09 & 4.04 & 2.346E-06 & 3.03 \\ \hline
  & 10 & 1.172E-04 & -- & 3.673E-06 & -- & 1.291E-04 & -- \\
1  & 20 & 1.338E-05 & 3.13 & 1.973E-07 & 4.22 & 1.449E-05 & 3.15 \\
  & 40 & 1.791E-06 & 2.90 & 1.278E-08 & 3.95 & 1.832E-06 & 2.98 \\
  & 80 & 2.199E-07 & 3.03 & 7.694E-10 & 4.05 & 2.194E-07 & 3.06 \\ \hline
  & 10 & 1.775E-04 & -- & 4.812E-06 & -- & 8.311E-05 & -- \\
2 & 20 & 2.018E-05 & 3.14 & 2.584E-07 & 4.22 & 9.561E-06 & 3.12 \\
  &40 & 2.638E-06 & 2.94 & 1.671E-08 & 3.95 & 1.204E-06 & 2.99 \\
  &80 & 3.146E-07 & 3.07 & 1.002E-09 & 4.06 & 1.445E-07 & 3.06 \\
\hline
\end{tabular}
\end{table}


\bibliographystyle{spmpsci}      
\bibliography{ref}

\begin{thebibliography}{10}
\providecommand{\url}[1]{{#1}}
\providecommand{\urlprefix}{URL }
\expandafter\ifx\csname urlstyle\endcsname\relax
  \providecommand{\doi}[1]{DOI~\discretionary{}{}{}#1}\else
  \providecommand{\doi}{DOI~\discretionary{}{}{}\begingroup
  \urlstyle{rm}\Url}\fi

\bibitem{AdFo2003}
Adams, R.A., Fournier, J.J.F.: Sobolev spaces, \emph{Pure and Applied
  Mathematics (Amsterdam)}, vol. 140, second edn.
\newblock Academic Press, Amsterdam (2003)

\bibitem{BrSc2008}
Brenner, S.C., Scott, L.R.: The mathematical theory of finite element methods,
  \emph{Texts in Applied Mathematics}, vol.~15, third edn.
\newblock Springer, New York (2008)

\bibitem{CQSS2017}
Chen, H., Qiu, W., Shi, K., Solano, M.: A superconvergent {HDG} method for the
  {M}axwell equations.
\newblock J. Sci. Comput. \textbf{70}(3), 1010--1029 (2017)

\bibitem{CPE2016}
Cockburn, B., Di~Pietro, D.A., Ern, A.: Bridging the hybrid high-order and
  hybridizable discontinuous {G}alerkin methods.
\newblock ESAIM Math. Model. Numer. Anal. \textbf{50}(3), 635--650 (2016)

\bibitem{PiAl2015}
Di~Pietro, D.A., Ern, A.: Hybrid high-order methods for variable-diffusion
  problems on general meshes.
\newblock C. R. Math. Acad. Sci. Paris \textbf{353}(1), 31--34 (2015)

\bibitem{PEL2014}
Di~Pietro, D.A., Ern, A., Lemaire, S.: An arbitrary-order and compact-stencil
  discretization of diffusion on general meshes based on local reconstruction
  operators.
\newblock Comput. Methods Appl. Math. \textbf{14}(4), 461--472 (2014)

\bibitem{FreeFem}
Hecht, F.: New development in freefem++.
\newblock J. Numer. Math. \textbf{20}(3-4), 251--265 (2012)

\bibitem{Lehrenfeld2010}
Lehrenfeld, C.: {H}ybrid discontinuous {G}alerkin methods for solving
  incompressible flow problems.
\newblock Master's Thesis, RWTH Aachen University  (2010)

\bibitem{LeSc2016}
Lehrenfeld, C., Sch\"oberl, J.: High order exactly divergence-free hybrid
  discontinuous {G}alerkin methods for unsteady incompressible flows.
\newblock Comput. Methods Appl. Mech. Engrg. \textbf{307}, 339--361 (2016)

\bibitem{Oikawa2015}
Oikawa, I.: A hybridized discontinuous {G}alerkin method with reduced
  stabilization.
\newblock J. Sci. Comput. \textbf{65}(1), 327--340 (2015)

\bibitem{Oikawa2016}
Oikawa, I.: Analysis of a reduced-order {HDG} method for the {S}tokes
  equations.
\newblock J. Sci. Comput. \textbf{67}(2), 475--492 (2016)

\bibitem{QJS2018}
Qiu, W., Shen, J., Shi, K.: An {HDG} method for linear elasticity with strong
  symmetric stresses.
\newblock Math. Comp. \textbf{87}(309), 69--93 (2018)

\bibitem{QiSh2016}
Qiu, W., Shi, K.: An {HDG} method for convection diffusion equation.
\newblock J. Sci. Comput. \textbf{66}(1), 346--357 (2016)

\bibitem{QiSh2016ns}
Qiu, W., Shi, K.: A superconvergent {HDG} method for the incompressible
  {N}avier-{S}tokes equations on general polyhedral meshes.
\newblock IMA J. Numer. Anal. \textbf{36}(4), 1943--1967 (2016)

\end{thebibliography}
\end{document}